\documentclass[11pt]{article}
\usepackage[a4paper, total={6in, 8in}]{geometry}
\def\ar{\rightarrow}
\usepackage{url}
\usepackage{graphicx,wrapfig,lipsum}
\usepackage[T1]{fontenc}
\usepackage[charter]{mathdesign}

\usepackage{caption}
\usepackage{amsmath,amsthm}
\usepackage{faktor}
\usepackage{tikz}
\usepackage{enumitem}
\usepackage{mathtools}
\usepackage{comment}
\usepackage{relsize}
\usepackage{exscale}
\usepackage{authblk}
\title{On the linearity of  fundamental groups of certain graphs of groups}

\author{D. Tsipa\thanks{%
		\protect\parbox[c]{0.17\linewidth}{%
			\protect\includegraphics[width=\linewidth]{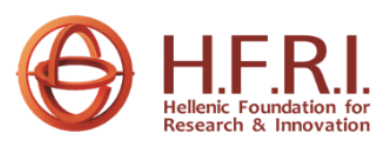}%
			
		}%
		\hfill%
		\protect\parbox[c]{0.77\linewidth}{%
			The research work was supported by the Hellenic Foundation for
			Research and Innovation (HFRI) under the 3rd Call for HFRI PhD
			Fellowships. (Fellowship Number: 5161)%
			
		}%
}}
\affil[ ]{{Department of Mathematics, University of the Aegean, Karlovassi, 832 00 Samos, Greece}}
\affil[ ]{\textit {dtsipa@aegean.gr}}

\date{}

\usepackage{tikz-cd}
\usepackage{pgfplots}
\usetikzlibrary{hobby}
\usepgfplotslibrary{polar}

\pgfplotsset{compat=newest}

\usepackage{tabularx}
\usepackage[hidelinks]{hyperref}

\newcommand{\C}{\mathbb{C}}

\newcommand{\Z}{\mathbb{Z}}

\newcommand*{\defeq}{\mathrel{\vcenter{\baselineskip0.5ex \lineskiplimit0pt
			\hbox{\scriptsize.}\hbox{\scriptsize.}}}%
	=}

\theoremstyle{definition}
\newtheorem{definition}{Definition}[section]
\newtheorem{theorem}{Theorem}[section]
\newtheorem{remark}{Remark}[section]
\newtheorem{proposition}{Proposition}[section]
\newtheorem{lemma}{Lemma}[section]

\newtheorem*{theorem-non}{Examples}
\newtheorem{corollary}{Corollary}[section]

\def\vdots{\vbox{\baselineskip=6pt \lineskiplimit=0pt 
		\kern6pt \hbox{.}\hbox{.}\hbox{.}}}
\def\gcd{{\rm gcd}}

\def\ar{\rightarrow}

\def\<{\langle}
\def\>{\rangle}

\begin{document}
	
\maketitle	
\vspace{-7mm}
\begin{abstract}
We study linearity over $\Z$ for fundamental groups of finite graphs of groups. Using embeddings into right-angled Artin groups and passage to finite-index subgroups, we establish sufficient conditions for integral linearity of multiple HNN extensions and of star graphs of groups with a free or free abelian central vertex group, free abelian leaf groups, and cyclic edge groups. We also prove that every amalgamated free product of two finitely generated free abelian groups is virtually a right-angled Artin group, and hence $\Z$-linear. Finally, we recover the characterization that a non-elementary generalized Baumslag--Solitar group is $\Z$-linear if and only if its modular homomorphism has image contained in $\{1,-1\}$.
\end{abstract}

\section{Introduction}
A group $G$ is \emph{linear} if it admits a faithful representation into $GL_n(R)$ for some positive integer $n$ and some integral domain $R$. It is \emph{$\mathbb{Z}$-linear} if it embeds in $GL_n(\mathbb{Z})$ for some $n$. Although the theory of linear groups is well developed, determining whether a given group is linear remains difficult. For example, Lubotzky gave necessary and sufficient conditions for a group to be linear over $\mathbb{C}$ in \cite{Lubotzky}, but verifying these conditions can be difficult. A fundamental problem is to determine which group-theoretic constructions preserve linearity.

The family of groups for which we will investigate linearity are fundamental groups of graphs of groups. In geometric group theory, a graph of groups is an object consisting of a collection of groups indexed by the vertices and edges of a graph, together with a family of monomorphisms of the edge groups into the vertex groups. Each finite connected graph of groups has an associated fundamental group, well defined up to isomorphism. It admits an action without inversions on a tree: the original graph of groups can be recovered from the quotient graph and the stabilizer subgroups. This theory, known as Bass--Serre theory, was developed by Hyman Bass and Jean-Pierre Serre (see \cite{Serre,Bass}).

 Nisnevi\v c was the first one who investigated the linearity of fundamental groups of graphs of groups (see \cite{Nisne}) and proved that the free product of linear groups is linear. Later, Shalen worked on the linearity of amalgamated free products of groups (see \cite{Shalen}) and proved that even if the amalgamated subgroup is cyclic, the product can be non-linear. In the positive direction, he proved that any group that can be built up from free
groups by successively forming free products with amalgamation, where the amalgamated subgroup is a maximal cyclic subgroup of each factor, is linear.

More recently, Berlai and de la Nuez Gonz\'alez \cite{Berlai2022} proved that a graph product over a finite graph of groups that are linear over an integral domain $R$ is linear over a suitable polynomial ring $R[\underline{t}]$, where $\underline{t}$ is a finite tuple of algebraically independent variables. In particular, any finite graph product of groups that are linear over $\mathbb{C}$ is again linear over $\mathbb{C}$, thereby resolving an open problem posed by Hsu and Wise in the complex case  (see \cite{HsuWise}).

Right-angled Artin groups are $\mathbb{Z}$-linear \cite{Davis,HsuWise}. Haglund and Wise \cite{HW} introduced special nonpositively curved cube complexes and proved that the fundamental group of a compact special cube complex embeds in a right-angled Artin group. Since $\mathbb{Z}$-linearity passes to subgroups and to finite extensions, virtually compact special groups are $\mathbb{Z}$-linear (see also \cite{Wise}). The theory of special cube complexes and virtual specialness played a central role in Agol's proof of the virtual Haken conjecture \cite{Agol}.

Metaftsis, Raptis and Varsos \cite{Mrv} established that an HNN-extension of a finitely generated free abelian group with proper associated subgroups is $\mathbb{Z}$-linear if and only if it is residually finite. Their argument uses Theorem 3.1 of \cite{Mrv}, concerning multiple HNN-extensions whose associated isomorphisms are restrictions of a common finite-order automorphism of the base group. In Section~4, we discuss a gap in the proof of that theorem and provide an independent proof. For HNN-extensions of $\mathbb{Z}^n$ whose associated subgroups both have finite index in the base group, Leary and Minasyan \cite{LearyMinasyan}, building on work of Andreadakis, Raptis and Varsos \cite{ARV}, proved that residual finiteness is equivalent to linearity over $\mathbb{Q}$. Their criterion involves both the associated subgroups and the rational matrix defining the associated isomorphism.

We use two embedding constructions to prove $\mathbb{Z}$-linearity for fundamental groups of certain rose and star graphs of groups. Lemma~\ref{Lemma1lin} treats stars with a free abelian centre and leaf groups that split as the incident edge group times a free abelian factor. Its proof constructs a finite-index subgroup with direct-factor attachments and embeds it in a right-angled Artin group by enlarging the centre. Lemma~\ref{Lemma2lin} treats multiple HNN-extensions of right-angled Artin groups whose associated isomorphisms extend to automorphisms induced by permutations of the defining graph. These constructions yield our results for stars with free abelian and free vertex groups, respectively.

For rose graphs, we give an independent proof of Theorem 3.1 of \cite{Mrv}: a multiple HNN-extension of a finitely generated free abelian group is $\mathbb{Z}$-linear when all associated isomorphisms are restrictions of a common finite-order automorphism of the base group. We also prove that an amalgamated free product of two finitely generated free abelian groups is $\mathbb{Z}$-linear (Theorem~\ref{amalfreeprod}) and, more strongly, virtually a right-angled Artin group (Remark~\ref{virtRAAG}). Finally, in Section~\ref{GBS}, we give a self-contained proof of the known characterization of $\mathbb{Z}$-linear generalized Baumslag--Solitar groups as the unimodular ones.

\section{Preliminaries}	
\vspace{2mm}
\subsection{Graphs of groups and their fundamental groups}
Let $\Gamma$ be a graph. We denote the set of vertices and the set of edges of $\Gamma$ with $V(\Gamma)$ and $E(\Gamma)$ respectively. We denote the inverse of an edge $e\in E(\Gamma)$ with $\bar{e}$ and the initial and terminal vertices of $e$ by $\alpha(e)$ and $\omega(e)$ respectively. Obviously $\alpha(e)=\omega(\bar{e})$ for every $e\in E(\Gamma)$.
\begin{definition}
	A\emph{ graph of groups} $(\mathcal{G},\Gamma)$ consists of a connected graph $\Gamma$, a vertex group $G_{v}$ for each vertex $v\in V(\Gamma)$, an edge group $G_{e}$ for each edge $e\in E(\Gamma)$, and monomorphisms $$\{\alpha_{e}:G_{e}\to G_{\alpha(e)} | e\in E(\Gamma)\}.$$ We require in addition that $G_{e}=G_{\bar{e}}$.
\end{definition}
\begin{remark}
	For an edge $e\in E(\Gamma)$, instead of the monomorphism $\alpha_{e}:G_{e}\to G_{\alpha(e)}$ we sometimes use the monomorphism $\omega_{e}:G_{e}\to G_{\omega(e)}$ defined by $\omega_{e}=\alpha_{\bar{e}}$.
\end{remark}
Denote by $F(\mathcal{G},\Gamma)$ the factor group of the free product of all groups $G_{v}$, $v\in V(\Gamma)$ and the free group $E$ with basis $\{t_{e} | e\in E(\Gamma)\}$ by the normal closure of the subset $\{t_{e}\alpha_{e}(g)t_{e}^{-1}(\alpha_{\bar{e}}(g))^{-1}\}\cup  \{t_{e}t_{\bar{e}}\}$, where $e\in E(\Gamma), g\in G_{e}$. Notice that $F({\cal G},\Gamma)$ is an HNN-extension of $*G_v$ with one stable letter for every edge pair. With this conjugation convention, traversal of an oriented edge $e$ is represented by $t_e^{-1}$. Thus any path in $\Gamma$ determines an element of $E$ and hence an element of $F({\cal G}, \Gamma)$. The group $F({\cal G},\Gamma)$ is called the \emph{path group} of the graph of groups.
\begin{definition}
	Let $(\mathcal{G},\Gamma)$ be a graph of groups and let $P$ be a vertex of the graph $\Gamma$. The \emph{fundamental group $\pi_{1}(\mathcal{G},\Gamma,P)$ of the graph of groups $(\mathcal{G},\Gamma)$ with respect to the vertex $P$ }is the subgroup of the group $F(\mathcal{G},\Gamma)$ consisting of all elements of the form $g_{0}t_{e_{1}}^{-1}g_{1}t_{e_{2}}^{-1}\dots t_{e_{n}}^{-1}g_{n}$, where $e_{1}e_{2}\dots e_{n}$ is a closed path in $\Gamma$ with initial and terminal vertex $P$, $g_{0}\in G_{P}$, $g_{i}\in G_{\omega(e_{i})}$, for all $i=1,\dots,n$.
\end{definition}
\begin{definition}
	Let $(\mathcal{G},\Gamma)$ be a graph of groups and let $T$ be a maximal subtree of the graph $\Gamma$. The \emph{fundamental group $\pi_{1}(\mathcal{G},\Gamma,T)$ of the graph of groups $(\mathcal{G},\Gamma)$ with respect to the subtree $T$} is the factor group of $F(\mathcal{G},\Gamma)$ by the normal closure of the set of elements $t_{e}$, $e\in E(T)$.
\end{definition}

The above two definitions yield isomorphic groups.
\begin{theorem}
	The fundamental groups $\pi_{1}(\mathcal{G},\Gamma
	,P)$ and $\pi_{1}(\mathcal{G},\Gamma,T)$ are isomorphic for any choice of the vertex $P$ and any choice of the maximal tree $T$ in the graph $\Gamma$.
\end{theorem}
\begin{proof}
	For the proof of this Theorem see \cite{Bog} or \cite{Serre}.
\end{proof}
The isomorphism class of these groups is denoted by $\pi_{1}(\mathcal{G},\Gamma)$.
\begin{remark}
It is easy to see that if  $\Gamma$ is a segment then $\pi_{1}(\mathcal{G},\Gamma)$ is isomorphic to the free product of the two groups associated to the vertices of $\Gamma$ amalgamated over the subgroups $\alpha_{e}(G_{e})$ and $\alpha_{\bar{e}}(G_{e})$, where $G_{e}$ is the group associated to the single edge $e$ of $\Gamma$. \par If $\Gamma$ is a loop then $\pi_{1}(\mathcal{G},\Gamma)$ is isomorphic to the HNN-extension with base the group $G_{P}$, where $P$ is the single vertex of $\Gamma$, and associated subgroups  $\alpha_{e}(G_{e})$ and $\alpha_{\bar{e}}(G_{e})$, where $G_{e}$ is again the group associated to the single edge $e$ of $\Gamma$. \par For an arbitrary graph of groups $({\mathcal{G},\Gamma})$ the fundamental group  $\pi_{1}(\mathcal{G},\Gamma)$  can be obtained from the fundamental group $\pi_{1}(\mathcal{G},T)$ (where $T$ is a maximal tree in $\Gamma$) by consecutive applications of HNN-extensions (the number of which is equal to the number of pairs of mutually inverse edges of $\Gamma$ not lying in the tree $T$), while the group $\pi_{1}(\mathcal{G},T)$ can be obtained from the fundamental group of a segment of groups by successive applications of the construction of an amalgamated product.  
\end{remark}
\subsection{Notation and conventions}
Throughout the following sections, all graphs of groups have finite connected underlying graphs and finitely generated vertex groups; all free groups and right-angled Artin groups are finitely generated. Moreover all groups, including abelian groups and finite cyclic quotients, are written multiplicatively, with identity element $1$. Also if $G$ is	a group then: \begin{enumerate}
	\item If $x,y\in G$, let $x^{y}=yxy^{-1}$.
	\item If $x,y\in G$, the commutator $[x,y]$ of $x$ and $y$ is defined to be $xyx^{-1}y^{-1}$.
	\item We denote by $G=\langle S_{G} | R_{G} \rangle $ a fixed presentation of $G$, where $ S_{G}$ is a generating set of $G$ and $R_{G}$ a set of defining relations of $G$.  If additionally, $G$ is  a right-angled Artin group then we assume that the presentation of $G$ is  the presentation on the standard RAAG generating set.
	\item If $S$ is a subset of $G$ then we denote by $S^{G}$ the normal closure of $S$ in $G$.
\end{enumerate}

\section{Embedding lemmas}

The following lemmas provide two embedding constructions used in the star-graph and rose-graph applications. The first passes to a finite-index subgroup and enlarges an abelian central vertex group diagonally. The second treats multiple HNN-extensions of right-angled Artin groups by extending partial actions to permutations of their defining graphs.

\begin{lemma}\label{Lemma1lin}
Let $A,B_1,\ldots,B_m$ be finitely generated free abelian groups and let $C_i\leq A$ for $1\leq i\leq m$. Let $G$ be the fundamental group of the star of groups with central vertex group $A$, leaf groups $C_i\times B_i$, and edge groups $C_i$, where each edge map is the given inclusion into $A$ and the first-factor inclusion into $C_i\times B_i$. Then $G$ has a finite-index subgroup which embeds in a finitely generated right-angled Artin group. In particular, $G$ is $\mathbb Z$-linear.
\end{lemma}
\begin{proof}
 For each $i$, let
\[
\widehat C_i=\{a\in A\mid a^n\in C_i\text{ for some }n\geq1\}
\]
be the saturation of $C_i$ in $A$. Each $\widehat C_i$ is a direct factor of $A$, and $\widehat C_i/C_i$ is finite. Choose a positive integer $M$ divisible by the exponents of all these finite quotients, and put $L=A^M=\{a^M:a\in A\}$. For any subgroup $U\leq A$, write $U^M=\{u^M:u\in U\}$.

There is a retraction $\rho:G\to A$ which is the identity on $A$ and sends $(c,b)\in C_i\times B_i$ to $c$. The subgroup $H=\rho^{-1}(L)$ has finite index $[A:L]$ in $G$. We describe its graph of groups using the action on the Bass--Serre tree (see \cite{Serre}). Put $Q=A/L$. There is one central vertex, since the image of $A$ in $Q$ is all of $Q$. Both the $i$th leaf group and its incident edge group have image $(C_iL)/L$. Thus the vertices and edges above this leaf and edge are indexed by the same finite set $Q/((C_iL)/L)$, with one edge incident to each lifted leaf. Consequently the underlying graph for $H$ is again a finite star.

Coset representatives can be chosen in $A$. Since $A$ is abelian, every lifted edge has image $L\cap C_i$ in the central group $L$, and the corresponding leaf group is isomorphic to $(L\cap C_i)\times B_i$. Moreover,
\[
L\cap C_i=\widehat C_i^{\,M}.
\]
Indeed, $\widehat C_i^{\,M}\subseteq L\cap C_i$ by the choice of $M$. Conversely, if $a^M\in C_i$ for $a\in A$, then $a\in\widehat C_i$. The power map $a\mapsto a^M$ is an isomorphism from $A$ onto $L$, so $\widehat C_i^{\,M}$ is a direct factor of $L$.

It remains to treat a finite star in which every attaching subgroup is a direct factor of the centre. Relabel its central group as $A$ and its leaf groups as $C_i\times B_i$, $1\leq i\leq q$. Set
\[
E=\prod_{j=1}^{q}A^{(j)},\qquad d:A\longrightarrow E,\quad d(a)=(a,\ldots,a).
\]
Choose a basis of each copy $A^{(i)}$ adapted to the direct factor $C_i^{(i)}$, and let
\[
P_i=\left(\prod_{j<i}A^{(j)}\right)\times C_i^{(i)}
       \times\left(\prod_{j>i}A^{(j)}\right).
\]
All the $P_i$ are generated by subsets of the resulting single basis of $E$, and
\begin{equation}\label{coordinateintersection}
d^{-1}(P_i)=C_i.
\end{equation}

Let $R$ be the fundamental group of the star with central group $E$, leaf groups $P_i\times B_i$, and the natural edge inclusions of $P_i$. After eliminating the repeated edge generators, its presentation consists of the commutators defining $E$, those defining the $B_i$, and the commutators between each basis generator of $B_i$ and each selected basis generator of $P_i$. Hence $R$ is a finitely generated right-angled Artin group.

Map the original central group by $d$, and map the $i$th leaf by $(c,b)\mapsto(d(c),b)$. These injective vertex maps agree on each edge group. At the central vertex, equation \eqref{coordinateintersection} says that the inverse image of the target edge group is exactly the source edge group. At the $i$th leaf, the inverse image of $P_i\times\{1\}$ is exactly $C_i\times\{1\}$. Therefore a reduced graph-of-groups path word remains reduced under these maps. The normal-form theorem for graphs of groups \cite{Serre} proves that the induced homomorphism into $R$ is injective.

Applying this construction to $H$ gives the required embedding. Finally, since right-angled Artin groups are $\mathbb Z$-linear and $\mathbb Z$-linearity passes to subgroups and finite extensions we get that $G$ is $\mathbb Z$-linear.
\end{proof}

\vspace{2mm}
For the second Lemma let  $K=\langle S_{K} | R_{K} \rangle$ be a finitely generated right-angled Artin group with $S_{K}=\{k_{j}\}$, $j=1,\dots, m$, and let $\Gamma_{K}$ be its defining graph. For $i\in \{1,\dots,n\}$, let $A_{i},B_{i}$ be subsets of $S_{K}$ with $|A_{i}|=|B_{i}|$ and let $f_{i} :A_{i} \to B_{i}$ be bijections.

 We assume that each $f_{i}$ extends to a permutation $\bar{f}_{i}$ of $S_{K}$ which is an automorphism of the graph $\Gamma_{K}$ (equivalently, $\bar{f}_{i}$ preserves adjacency and non-adjacency between standard generators). Then $\bar{f}_{i}$ induces an automorphism of $K$ of finite order, which we denote again by $\bar{f}_{i}$, and we set $r_{i}=ord(\bar{f}_{i})$. In particular, $f_{i}$ extends to an isomorphism (also denoted $f_{i}$) $f_{i} :H_{i} \to f_{i}(H_{i})$, where $H_{i}$ and $f_{i}(H_{i})$ are the subgroups of $K$ generated by $A_{i}$ and $B_{i}$ respectively, and $\bar{f}_{i}|_{H_{i}}=f_{i}$.
\par For each $i=1,\dots,n$ we can define the HNN-extension $$K_{i}^{*}=\langle S_{K},t_{i} | R_{K}, t_{i}a_{i}t_{i}^{-1}=f_{i}(a_{i}), a_{i}\in A_{i}\rangle$$ with $K$ as the base group, $t_{i}$ as stable letter and $H_{i}, f_{i}(H_{i})$ as associated subgroups. In each $K_{i}^{*}$ the stable letter $t_{i}$ acts on $S_{K}$ as the partial permutation $f_{i}$, and $\bar{f}_{i}$ extends this action to an automorphism of $K$ of finite order. Notice that, since each $\bar{f}_{i}$ is a permutation of the finite set $S_{K}$, the subgroup of $Aut(K)$ generated by $\bar{f}_{1},\dots,\bar{f}_{n}$ is always finite.

\begin{lemma}\label{Lemma2lin}
	Let  $K$, $K_{i}^{*}$ be the groups described above and let $D$ be the (finite) subgroup of $Aut(K)$ generated by the automorphisms $\bar{f}_{1},\dots,\bar{f}_{n}$. Then the group $$G=\langle S_{K}, t_{1},\dots, t_{n}\hspace{1mm}|\hspace{1mm} R_{K}, t_{i}a_{i}t_{i}^{-1}=f_{i}(a_{i}), a_{i}\in A_{i}, i=1,\dots,n \rangle $$ is $\mathbb{Z}$-linear.
\end{lemma}
\begin{proof}
	First of all, notice that $G$ is a multiple HNN-extension with base group $K$, stable letters $t_{1},\dots,t_{n}$ and associated subgroups $H_{i}$, $f_{i}(H_{i})$, $i=1,\dots,n$. That is, $G$ can be considered as the fundamental group of the following graph of groups. \\
	\begin{center}
		\begin{tikzpicture}[scale=1.02]
			\begin{polaraxis}[grid=none, axis lines=none]
				\addplot[mark=none,domain=0:360,samples=200] {cos(x*3)};
			\end{polaraxis}
			\node at (3.7,3.8) (nodeP) {\large{$K$}};
			\fill (3.43cm,3.43cm) circle (1.5pt); 
			\node at (6.2,4.3) (nodeQ) {\small{$t_{1}a_{1}t_{1}^{-1}=f_{1}(a_{1}), a_{1}\in A_{1}$}};
			\node at (0.02,4.5) (nodeR) {\small{$t_{2}a_{2}t_{2}^{-1}=f_{2}(a_{2}), a_{2}\in A_{2}$}};
			\node at (-0.21,2.4) (nodeS) {\small{$t_{3}a_{3}t_{3}^{-1}=f_{3}(a_{3}), a_{3}\in A_{3}$}};
			\fill (4.73cm,2.38cm) circle (1.5pt);
			\fill (4.28cm,2.13cm) circle (1.5pt);
			\fill (3.73cm,1.93cm) circle (1.5pt);
		\end{tikzpicture}
	\end{center}
\vspace{3mm}
Let $C$ be the group $$C=\langle S_{K},\xi_{1},\dots,\xi_{n}\hspace{1mm}|\hspace{1mm} R_{K},\ \xi_{i}k\xi_{i}^{-1}=\bar{f}_{i}(k),\ k\in S_{K},\ i=1,\dots,n\rangle .$$ Since each $\bar{f}_{i}$ is an automorphism of $K$, $C$ is the semidirect product $K\rtimes F_{n}$, where $F_{n}$ is the free group on $\xi_{1},\dots,\xi_{n}$ acting on $K$ via $\xi_{i}\mapsto \bar{f}_{i}$; in particular $K$ is a subgroup of $C$. Let $\tilde{G}$ be the multiple HNN-extension of $C$ with stable letters $\zeta_{1},\dots,\zeta_{n}$, where $\zeta_{i}$ centralizes the subgroup $H_{i}$, that is
	$$\tilde{G}=\langle S_{K},\xi_{1},\dots,\xi_{n},\zeta_{1},\dots,\zeta_{n}\hspace{1mm}|\hspace{1mm} R_{K},\ \xi_{i}k\xi_{i}^{-1}=\bar{f}_{i}(k),\ [\zeta_{i},a_{i}]=1\rangle ,$$ where $k\in S_{K}$, $a_{i}\in A_{i}$ and $i=1,\dots,n$.
	\par The map $\phi: G \to \tilde{G}$ with $\phi(k)=k$ for every $k\in S_{K}$ and $\phi(t_{i})=\xi_{i}\zeta_{i}$ for $i=1,\dots,n$ is a homomorphism, since for every $a_{i}\in A_{i}$ we have $$\phi(t_{i}a_{i}t_{i}^{-1})=\xi_{i}\zeta_{i}a_{i}\zeta_{i}^{-1}\xi_{i}^{-1}=\xi_{i}a_{i}\xi_{i}^{-1}=\bar{f}_{i}(a_{i})=f_{i}(a_{i}).$$
	We show that $\phi$ is injective. Let $1\neq g\in G$. If $g\in K$ then $\phi(g)=g\neq 1$, since $K$ embeds in $C$ and $C$ embeds in $\tilde{G}$. Otherwise, by Britton's Lemma (see \cite{Britton}) we can write $$g=x_{0}t_{i_{1}}^{\epsilon_{1}}x_{1}t_{i_{2}}^{\epsilon_{2}}\cdots t_{i_{l}}^{\epsilon_{l}}x_{l},\qquad l\geq 1,\ x_{j}\in K,\ \epsilon_{j}=\pm1,$$ where this expression contains no pinch, that is, there is no $j$ with $i_{j}=i_{j+1}=i$ and either $\epsilon_{j}=1=-\epsilon_{j+1}$ and $x_{j}\in H_{i}$, or $\epsilon_{j}=-1=-\epsilon_{j+1}$ and $x_{j}\in f_{i}(H_{i})$. Since $\phi(t_{i})=\xi_{i}\zeta_{i}$ and $\phi(t_{i}^{-1})=\zeta_{i}^{-1}\xi_{i}^{-1}$, we have $$\phi(g)=c_{0}\zeta_{i_{1}}^{\epsilon_{1}}c_{1}\zeta_{i_{2}}^{\epsilon_{2}}\cdots \zeta_{i_{l}}^{\epsilon_{l}}c_{l},\qquad c_{j}=\xi_{i_{j}}^{-\delta_{j}}\,x_{j}\,\xi_{i_{j+1}}^{\delta'_{j+1}}\in C,$$ where the formula for $c_j$ applies for $1\leq j<l$. For $1\leq j\leq l$, put $\delta_j=(1-\epsilon_j)/2$ and $\delta'_j=(1+\epsilon_j)/2$. The endpoint coefficients are $c_0=x_0\xi_{i_1}^{\delta'_1}$ and $c_l=\xi_{i_l}^{-\delta_l}x_l$. We claim that this expression contains no pinch with respect to the HNN-structure of $\tilde{G}$ over $C$. A pinch could only occur at some $j$ with $i_{j}=i_{j+1}=i$ and $\epsilon_{j}=-\epsilon_{j+1}$.
	\begin{itemize}
		\item If $\epsilon_{j}=1$ and $\epsilon_{j+1}=-1$, then $c_{j}=x_{j}$, and $x_{j}\notin H_{i}$ by the choice of the expression of $g$.
		\item If $\epsilon_{j}=-1$ and $\epsilon_{j+1}=1$, then $c_{j}=\xi_{i}^{-1}x_{j}\xi_{i}=\bar{f}_{i}^{-1}(x_{j})$. Since $\bar{f}_{i}(H_{i})=f_{i}(H_{i})$ and $x_{j}\notin f_{i}(H_{i})$, we have $c_{j}\notin H_{i}$.
	\end{itemize}
	Hence, by Britton's Lemma, $\phi(g)\neq 1$ and $\phi$ is a monomorphism.
	\par Since $D$ acts on the finite set $S_{K}$ by permutations, for every $\sigma\in D$ and $k\in S_{K}$ the element $\sigma(k)$ is again a generator of $K$. Let $\theta:\tilde{G}\twoheadrightarrow D$ be the epimorphism with $\theta(\xi_{i})=\bar{f}_{i}$, $\theta(\zeta_{i})=1$ and $\theta(k)=1$ for $k\in S_{K}$ (it is well defined, since all relators of $\tilde{G}$ are mapped to $1$), and let $N=Ker\theta$, a subgroup of index $|D|$ in $\tilde{G}$. Choose a Schreier transversal $U$ for the kernel of the restriction $\theta|_{F_{n}}:F_{n}\to D$, consisting of reduced words in $\xi_{1},\dots,\xi_{n}$ (so $1\in U$ and $U$ is closed under taking prefixes). Since $\theta(k)=\theta(\zeta_{i})=1$, $U$ is also a Schreier transversal for $N$ in $\tilde{G}$. For $g\in\tilde{G}$ we denote by $\bar{g}$ its representative in $U$, and for $u\in U$ we write $\sigma_{u}=\theta(u)\in D$. By the Reidemeister-Schreier method, $N$ is generated by the non-trivial elements among
	$$k^{u}=uku^{-1},\qquad \zeta_{i}^{u}=u\zeta_{i}u^{-1},\qquad \gamma_{u,i}=u\xi_{i}(\overline{u\xi_{i}})^{-1},\qquad k\in S_{K},\ u\in U,\ i=1,\dots,n,$$
	and the defining relators of $N$ are the relators $uru^{-1}$, $u\in U$, $r\in R_{\tilde{G}}$, rewritten in these generators:
	\begin{itemize}
		\item[$(R_{1}^{u})$] $[k^{u},k'^{u}]$, for every relator $[k,k']\in R_{K}$;
		\item[$(R_{2}^{u})$] $[\zeta_{i}^{u},a_{i}^{u}]$, for $a_{i}\in A_{i}$;
		\item[$(R_{3}^{u})$] $\gamma_{u,i}\,k^{v}\,\gamma_{u,i}^{-1}=(\bar{f}_{i}(k))^{u}$, for $k\in S_{K}$, where $v=\overline{u\xi_{i}}$.
	\end{itemize}
	Recall that $\gamma_{u,i}=1$ exactly when $u\xi_{i}$ (freely reduced) belongs to $U$, and that the non-trivial elements $\gamma_{u,i}$ form a free basis of the free group $Ker(\theta|_{F_{n}})$.
	\par We claim that, using the relations $(R_{3}^{u})$ with $\gamma_{u,i}=1$, we can eliminate all generators $k^{u}$ with $u\neq 1$ by means of the equalities $$k^{u}=\sigma_{u}(k),\qquad k\in S_{K},\ u\in U.$$ We argue by induction on the length of $u$. For $u=1$ there is nothing to prove. If $u=u'\xi_{i}$ with $u'\in U$, then $\gamma_{u',i}=1$ and $(R_{3}^{u'})$ reads $k^{u}=(\bar{f}_{i}(k))^{u'}=\sigma_{u'}\bar{f}_{i}(k)=\sigma_{u}(k)$. If $u=u'\xi_{i}^{-1}$ with $u'\in U$, then $u\xi_{i}=u'\in U$, so $\gamma_{u,i}=1$ and $(R_{3}^{u})$ reads $k^{u'}=(\bar{f}_{i}(k))^{u}$; applying this to $\bar{f}_{i}^{-1}(k)$ we obtain $k^{u}=\sigma_{u'}\bar{f}_{i}^{-1}(k)=\sigma_{u}(k)$. Since every $u\neq1$ in $U$ is obtained from a unique shorter element of $U$ in this way, each of these relations is used exactly once, to eliminate the generators $k^{u}$, $k\in S_{K}$, by Tietze transformations. After this elimination:
	\begin{itemize}
		\item the relators $(R_{1}^{u})$ become $[\sigma_{u}(k),\sigma_{u}(k')]$, which belong to $R_{K}$ because $\sigma_{u}$ is an automorphism of the defining graph of $K$;
		\item the relators $(R_{2}^{u})$ become $[\zeta_{i}^{u},\sigma_{u}(a_{i})]$, $a_{i}\in A_{i}$, which are commutators of generators, since $\sigma_{u}(a_{i})\in S_{K}$;
		\item the remaining relators $(R_{3}^{u})$, i.e. those with $\gamma_{u,i}\neq 1$, become $\gamma_{u,i}\sigma_{v}(k)\gamma_{u,i}^{-1}=\sigma_{u}\bar{f}_{i}(k)$. Since $\theta(v)=\theta(u\xi_{i})$, we have $\sigma_{v}=\sigma_{u}\bar{f}_{i}$, so these relators are $[\gamma_{u,i},\sigma_{v}(k)]$, $k\in S_{K}$; as $k$ runs through $S_{K}$, so does $\sigma_{v}(k)$. Thus they say that each $\gamma_{u,i}\neq 1$ commutes with every generator of $K$.
	\end{itemize}
	Consequently $$N=\langle S_{K},\ \zeta_{i}^{u},\ \gamma_{u,i}\ (\gamma_{u,i}\neq 1)\hspace{1mm}|\hspace{1mm} R_{K},\ [\zeta_{i}^{u},\sigma_{u}(a_{i})],\ [\gamma_{u,i},k]\rangle ,$$ where $k\in S_{K}$, $a_{i}\in A_{i}$, $u\in U$ and $i=1,\dots,n$. This is a presentation on finitely many generators in which every relator is a commutator of two generators. Hence $N$ is a finitely generated right-angled Artin group and therefore $\mathbb{Z}$-linear (\cite{Davis,HsuWise}). Since $N$ has finite index in $\tilde{G}$, the group $\tilde{G}$ is $\mathbb{Z}$-linear (by inducing the representation), and hence so is its subgroup $\phi(G)\cong G$.
\end{proof}
\begin{remark}
For a single stable letter, the conclusion of Lemma~\ref{Lemma2lin} follows from Hsu and Leary \cite[Theorem~1]{HsuLeary}, under weaker hypotheses. More precisely, let $K=\langle S_{K}\mid R_K\rangle$ be a right-angled Artin group with its standard presentation, let $S_1,S_2\subseteq S_{K}$, and let $f:S_1\to S_2$ be a bijection inducing an isomorphism between the corresponding induced subgraphs of the defining graph of $K$. Then the HNN-extension
\[
G=\langle S_{K},t\mid R_K,\ t a t^{-1}=f(a)\ (a\in S_1)\rangle
\]
virtually embeds in a right-angled Artin group and is therefore $\mathbb{Z}$-linear. This result does not require $f$ to extend to an automorphism of the whole defining graph. Lemma~\ref{Lemma2lin} treats several stable letters under the additional hypothesis that each associated partial map extends to an automorphism of the defining graph.
\end{remark}
Here we present a Corollary of the above Lemma.
\begin{corollary}
	Let $G$ be a group with presentation $$G=\langle x,y \hspace{1mm}|\hspace{1mm}[x^{n},y^{m}]=1\rangle,\hspace{2mm} n,m\in\mathbb{Z}_{>0}.$$ Then $G$ is $\mathbb{Z}$-linear.
\end{corollary}
\begin{proof}
	Let $Q_n=\langle u\mid u^n=1\rangle$, and let $N$ be the kernel of the epimorphism $G\twoheadrightarrow Q_n$ that maps $x$ to $u$ and $y$ to $1$. Using the Schreier transversal $\{1,x,\dots,x^{n-1}\}$, the Reidemeister-Schreier method gives the presentation $$N=\langle z, t_{0}, t_{1}, \dots, t_{n-1}\hspace{1mm}|\hspace{1mm} [t_{i}^{m},z]=1,\ i=0,\dots,n-1\rangle$$ where $z=x^{n}$ and $t_{i}=x^{i}yx^{-i}$, $i\in \{0,\dots,n-1\}$.
	If $m=1$, this presentation gives $N\cong\mathbb{Z}\times F_n$, which is a right-angled Artin group, and the result follows. Assume henceforth that $m\geq2$. For each $i\in \{0,\dots,n-1\}$ we have  $$[t_{i}^{m},z]=1\Leftrightarrow t_{i}^{m}zt_{i}^{-m}=z\Leftrightarrow\underbrace{t_{i}\dots t_{i}}_{m \text{ times}}z\underbrace{t_{i}^{-1}\dots t_{i}^{-1}}_{m \text{ times}}=z.$$
	Let  \hspace{2mm}
	$\begin{cases}
		r_{i,1} & =  t_{i}zt_{i}^{-1}
		\\ r_{i,2} & =  t_{i}r_{i,1}t_{i}^{-1}\\ & \vdots  \\
		r_{i,m-1} &=  t_{i}r_{i,m-2}t_{i}^{-1}
	\end{cases}   $.
	\vspace{3mm}

	Using Tietze transformations we add the elements $\{r_{i,j}\}$ to the presentation of $N$ as new generators, and the relation $t_{i}^{m}zt_{i}^{-m}=z$ can be replaced by the set of relations $$\{t_{i}zt_{i}^{-1}=r_{i,1},  t_{i}r_{i,1}t_{i}^{-1}=r_{i,2},\dots, t_{i}r_{i,m-2}t_{i}^{-1}=r_{i,m-1},t_{i}r_{i,m-1}t_{i}^{-1}=z \}$$ for all $i\in \{0,\dots,n-1\}$.

	Thus $N$ is a multiple HNN-extension of the free group $K$ with basis $S_{K}=\{z\}\cup\{r_{i,j}\mid 0\leq i\leq n-1,\ 1\leq j\leq m-1\}$ (a right-angled Artin group without relators), with stable letters $t_{0},\dots,t_{n-1}$, and the action of $t_{i}$ on $A_{i}=\{z,r_{i,1},\dots,r_{i,m-1}\}$ extends to the automorphism of $K$ defined by the cyclic permutation $(z\hspace{2mm} r_{i,1}\hspace{2mm} r_{i,2} \cdots r_{i,m-1} )$ of $S_{K}$. Since $K$ is free, every permutation of $S_{K}$ is an automorphism of its defining graph. Hence, by Lemma \ref{Lemma2lin}, the group $N$, and therefore $G$, is $\mathbb{Z}$-linear.
\end{proof}
\section{Applications to rose graphs}

As we already stated, an HNN-extension can be considered as the fundamental group of a graph of groups $(\mathcal{G},\Gamma)$, where  $\Gamma$ is a loop,  the  base group is the group $G_{P}$, where $P$ is the single vertex of $\Gamma$, and associated subgroups  $\alpha_{e}(G_{e})$ and $\alpha_{\bar{e}}(G_{e})$, where $G_{e}$ is the group associated to the single edge $e$ of $\Gamma$. 

In the following we discuss a gap in the proof of Theorem 3.1 of \cite{Mrv}. In that proof, $K$ is a finitely generated free abelian group, $\varphi$ an automorphism of $K$ of finite order $\nu$, $A$ a subgroup of $K$, and $\tilde{G}$ is generated by $K$, $\zeta$ and $\xi_{1},\dots,\xi_{n}$ subject to the relations $\xi_{i}k\xi_{i}^{-1}=\varphi(k)$, $k\in K$, and $[\zeta,a]=1$, $a\in A$; then $H$ denotes the kernel of the epimorphism $\tilde{G}\twoheadrightarrow Q_{\nu}=\langle u\mid u^{\nu}=1\rangle$ that maps every $\xi_i$ to $u$ and maps $K$ and $\zeta$ to $1$. With $k_{sj}=\xi_{1}^{s}k_{j}\xi_{1}^{-s}$, the authors state that, since $\varphi$ is an automorphism, every $\varphi(k_{sj})$ is a generating set of a copy of $K$ for every $s=0,\dots,\nu-1$, and so each generating set $\{k_{sj}\}$ can be replaced by $\{\varphi(k_{sj})\}$. However, $\xi_{1}$ normalizes $K$, so the subgroups $\xi_{1}^{s}K\xi_{1}^{-s}$, $s=0,\dots,\nu-1$, all coincide with $K$, and the relations $k_{(s+1)j}=\varphi(k_{sj})$ express every $k_{sj}$ as a word in one fixed basis $\{k_{0j}\}$ of $K$. Consequently, the relators $[\zeta,a]$ give in $H$ the relators $[z_{s},\varphi^{s}(a)]$, $a\in A$, $s=0,\dots,\nu-1$, where $z_{s}=\xi_{1}^{s}\zeta\xi_{1}^{-s}$, and these are commutators of generators only if all the subgroups $\varphi^{s}(A)$ are spanned by subsets of one common basis of $K$. This need not be the case. For example, let $K=\mathbb{Z}^{2}$ with basis $\{e_{1},e_{2}\}$, $A=\langle e_{1}\rangle$ and $\varphi(e_{1})=e_{2}$, $\varphi(e_{2})=e_{1}^{-1}e_{2}^{-1}$, so that $\varphi$ has order $3$. Then $z_{0},z_{1},z_{2}$ commute with $e_{1}$, $e_{2}$ and $e_{1}e_{2}$ respectively; since these elements generate three distinct cyclic direct factors of $\mathbb{Z}^{2}$, there is no basis of $K$ in which all three are spanned by basis elements. Consequently, the stated Tietze transformations in \cite{Mrv} do not establish a right-angled Artin presentation for $H$. (The group $G$ of this example is nevertheless $\mathbb{Z}$-linear, for instance by Lemma \ref{Lemma2lin} applied with the automorphism of $K$ that exchanges $e_{1}$ and $e_{2}$; the point is that the argument of \cite{Mrv} does not prove it.)

 We now give an independent proof of Theorem 3.1 of \cite{Mrv}, using Lemma \ref{Lemma1lin}. In contrast to \cite{Mrv}, we use a separate letter $\zeta_{i}$ for each stable letter $t_{i}$, and no assumption on $A$ is needed.

\begin{samepage}
\begin{theorem}\label{MRVthm}
	Let $K$ be a finitely generated free abelian group, let $A$ be a subgroup of $K$, and let $f$ be an automorphism of $K$ of finite order. Then the multiple HNN-extension $$G=\langle S_{K}, t_{1},\dots,t_{n}\hspace{1mm}|\hspace{1mm}R_{K},\ t_{i}at_{i}^{-1}=f(a),\ a\in A,\ i=1,\dots,n\rangle$$ is $\mathbb{Z}$-linear.
\end{theorem}
\end{samepage}
\begin{proof}
Let $r=\operatorname{ord}(f)$ and let $C=K\rtimes_f\langle\xi\rangle$, where $\langle\xi\rangle\cong\mathbb Z$ acts on $K$ by $\xi k\xi^{-1}=f(k)$. Consider the multiple HNN-extension of $C$ with stable letters $\zeta_{1},\dots,\zeta_{n}$, each centralizing $A$:
\[
\widetilde G=\langle C,\zeta_{1},\dots,\zeta_{n}\mid [\zeta_{i},a]=1\ (a\in A,\ i=1,\dots,n)\rangle.
\]
The map $\phi:G\to\widetilde G$ given by $\phi(k)=k$ for $k\in K$ and $\phi(t_{i})=\xi\zeta_{i}$, $i=1,\dots,n$, is a homomorphism, since
\[
(\xi\zeta_{i})a(\xi\zeta_{i})^{-1}=\xi a\xi^{-1}=f(a)\qquad (a\in A).
\]
To prove injectivity, let $1\neq g\in G$. The group $K$ embeds in $C$ and hence in $\widetilde G$, so we may assume that $g\notin K$ and write $g$ in Britton-reduced form $g=x_{0}t_{i_{1}}^{\epsilon_{1}}x_{1}\cdots t_{i_{l}}^{\epsilon_{l}}x_{l}$ with $l\geq1$ and $x_{j}\in K$. Substituting $\xi\zeta_{i}$ for $t_{i}$ and $\zeta_{i}^{-1}\xi^{-1}$ for $t_{i}^{-1}$, and collecting the intervening elements of $C$, we obtain an expression $\phi(g)=c_{0}\zeta_{i_{1}}^{\epsilon_{1}}c_{1}\cdots\zeta_{i_{l}}^{\epsilon_{l}}c_{l}$ with $c_{j}\in C$. A $\zeta$-pinch can only occur at some $j$ with $i_{j}=i_{j+1}=i$ and $\epsilon_{j}=-\epsilon_{j+1}$. If $\epsilon_{j}=1$ and $\epsilon_{j+1}=-1$, then $c_{j}=x_{j}$, which does not lie in $A$ since the subword $t_{i}x_{j}t_{i}^{-1}$ of $g$ is not a pinch. If $\epsilon_{j}=-1$ and $\epsilon_{j+1}=1$, then $c_{j}=\xi^{-1}x_{j}\xi=f^{-1}(x_{j})$, which does not lie in $A$ since $x_{j}\notin f(A)$. Hence the expression of $\phi(g)$ contains no pinch, and Britton's lemma shows that $\phi(g)\neq1$. Thus $\phi$ is injective.

Let $Q_r=\langle u\mid u^r=1\rangle$. Define $\theta:\widetilde G\to Q_r$ by $\theta(\xi)=u$ and $\theta(k)=\theta(\zeta_{i})=1$ for every $k\in K$ and $i=1,\dots,n$. With transversal $\{1,\xi,\ldots,\xi^{r-1}\}$, Reidemeister--Schreier rewriting gives the following presentation for $N=\ker\theta$:
\[
N=\left\langle S_K,z,\zeta_{i,j}\ (1\leq i\leq n,\ 0\leq j<r)\ \middle|\
R_K,\ [z,k]=1,\ [\zeta_{i,j},f^j(a)]=1
\right\rangle,
\]
where $z=\xi^r$, $\zeta_{i,j}=\xi^j\zeta_{i}\xi^{-j}$, $k\in S_K$ and $a\in A$. Indeed, the conjugate copies of $K$ are identified with $K$ by $\xi^jk\xi^{-j}=f^j(k)$, and $f^r=1$ gives $[z,k]=1$. The conjugates of $[\zeta_{i},a]=1$ by $\xi^{j}$ give exactly the displayed relations involving $\zeta_{i,j}$. This description also applies when $r=1$.

Thus $N$ is the fundamental group of a star with free abelian centre $K\times\langle z\rangle$, edge subgroups $f^j(A)\leq K$, and $nr$ leaf groups $f^j(A)\times\langle\zeta_{i,j}\rangle$. Lemma \ref{Lemma1lin} applies without any primitivity assumption on $A$, so $N$ is $\mathbb Z$-linear. Since $N$ has finite index in $\widetilde G$ and $G$ embeds in $\widetilde G$, the result follows.
\end{proof}
The following proposition is a direct application of Lemma \ref{Lemma1lin}.
\begin{proposition}\label{Propositionfree}
	Let $F_{n}$, $n\geq 1$ be a free group of finite rank $n$ with basis $\{x_{0},x_{1},\dots,x_{n-1}\}$ and $\phi$ be an automorphism of $F_{n}$ such that $\phi(x_{i})=x_{0}^{-i}x_{i}x_{0}^{i}$ for all $i=0,1,\dots,n-1$. Then the HNN-extension $$G=\langle t, F_{n} \hspace{1mm} |\hspace{1mm} tx_{i}t^{-1}=\phi(x_{i}), i=0,1,\dots,n-1 \rangle$$ is $\mathbb{Z}$-linear.
\end{proposition}
\begin{proof}
The defining relations are equivalently
\[
G=\langle x_0,t,x_1,\ldots,x_{n-1}\mid [x_0,t]=1,\ [x_0^it,x_i]=1\ (1\leq i\leq n-1)\rangle.
\]
Indeed, $tx_it^{-1}=x_0^{-i}x_ix_0^i$ is equivalent to $[x_0^it,x_i]=1$, and the relation for $i=0$ is $[t,x_0]=1$. This is a star of groups with free abelian centre $\langle x_0,t\mid[x_0,t]=1\rangle$, edge groups $C_i=\langle x_0^it\rangle$, and leaf groups $C_i\times\langle x_i\rangle$. Lemma \ref{Lemma1lin} proves the assertion.
\end{proof}

\begin{theorem}\label{rosecyclic}
	Let $G$ be a multiple HNN-extension with base a finitely generated free abelian group $K$, in which each stable letter $t_{i}$ centralizes a cyclic subgroup $\langle w_{i}\rangle$ of $K$, i.e. $G$ has a presentation
	$$G=\langle S_{K},t_{1},\dots,t_{n}\hspace{1mm} | \hspace{1mm} R_{K}, [t_{i},w_{i}]=1\rangle$$
	where $w_{i}$ are non-trivial words in $S_{K}$ for $i\in \{1,\dots,n\}$. Then $G$ is $\mathbb{Z}$-linear.
\end{theorem}

\begin{proof}
For each $i$, put $C_i=\langle w_i\rangle\leq K$ and let $B_i=\langle t_i\rangle\cong\mathbb Z$. The displayed presentation identifies $G$ with the fundamental group of the star having centre $K$, leaf groups $C_i\times B_i$, and the natural edge inclusions of $C_i$. Lemma \ref{Lemma1lin} therefore gives $\mathbb Z$-linearity, whether or not $w_i$ is primitive in $K$.
\end{proof}

\section{Applications to star graphs}
We prove $\mathbb{Z}$-linearity for star graphs of groups with infinite cyclic edge groups and free abelian leaf groups. We treat a free abelian central vertex group using Lemma~\ref{Lemma1lin}, and a free central vertex group using Lemma~\ref{Lemma2lin}. 
\subsection{A free abelian central vertex group}
\begin{definition}
Let $G$ be a free abelian group with basis $S=\{s_1,\dots,s_n\}$. An element $w\in G$ is \emph{primitive} if it belongs to a basis of $G$.
\end{definition}
\begin{remark}
Write $w=s_1^{\nu_1}\cdots s_n^{\nu_n}$, where the exponents are collected with respect to the distinct basis elements of $S$. Then $w$ is primitive if and only if $\gcd(\nu_1,\dots,\nu_n)=1$. Indeed, this condition is equivalent to the row $[\nu_1\ \cdots\ \nu_n]$ extending to an integral matrix of determinant $\pm1$, whose rows give a basis of $G$ containing $w$ (see \cite{Newman}).
\end{remark}

\begin{theorem}\label{prop1}
	Let ($\mathcal{G},S_{m}$)  be a star graph of groups with free abelian vertex
	groups and infinite cyclic edge groups. For each edge $e_{i}$
	let $\alpha_{e_{i}}$
	and $\omega_{e_{i}}$ be the monomorphisms of the edge group to the central and the external vertex group respectively.
	For each $i=1,\dots,m$, choose a generator $c_i$ of the edge group and assume that $\omega_{e_i}(c_i)$ is primitive in the corresponding leaf group. Then the fundamental
	group of the graph of groups $G=\pi_{1}(\mathcal{G},S_{m})$ is $\mathbb{Z}$-linear.
\end{theorem}
\begin{proof}
Let $K_0$ be the central vertex group and let $K_i$ be the $i$th leaf group. Choose a generator $c_i$ of the $i$th edge group, and put $a_i=\omega_{e_i}(c_i)$ and $w_i=\alpha_{e_i}(c_i)\in K_0$. By hypothesis $a_i$ is primitive in $K_i$, so there is a finitely generated free abelian group $B_i$ such that
\[
K_i=\langle a_i\rangle\times B_i.
\]
The edge identification $a_i=w_i$ identifies the given star with the star having centre $K_0$, edge subgroups $C_i=\langle w_i\rangle\leq K_0$, and leaf groups $C_i\times B_i$. Lemma \ref{Lemma1lin} now applies. No primitivity assumption on $w_i$ in the central group is required.
\end{proof}

The groups in Proposition \ref{Propositionfree}, Theorem \ref{rosecyclic} and Theorem \ref{prop1} are fundamental groups of finite graphs of free abelian groups with cyclic edge groups. Wise \cite{WiseTubular} characterised when such a group acts freely on a CAT(0) cube complex in the case where all vertex groups are isomorphic to $\mathbb{Z}^{2}$ (tubular groups), and Woodhouse \cite{Woodhouse} showed that a tubular group is virtually special if and only if it acts freely on a finite-dimensional CAT(0) cube complex. When all vertex groups have rank two, the conclusion of Theorem \ref{prop1} also follows from Button's theorem that tubular groups with underlying graph a tree have a finite-index subgroup embedding in a right-angled Artin group \cite{ButtonTrees}. The argument above allows arbitrary finite ranks under the stated leaf-primitivity assumption.

\subsection{A free central vertex group}
We now use Lemma \ref{Lemma2lin} to treat a free central vertex group and free abelian leaf groups.

\begin{theorem}\label{propo2}
	Let ($\mathcal{G},S_{m}$)  be a star graph of groups with a free group
	associated to the central vertex, free abelian groups associated to the external vertices and infinite cyclic edge groups. For each edge $e_{i}$ let $\alpha_{e_{i}}$ and $\omega_{e_{i}}$ be the monomorphisms
	of the edge group to the central and the external vertex group respectively. For each $i=1,\dots,m$, choose a generator $c_i$ of the edge group and assume that $\omega_{e_i}(c_i)$ is primitive in the corresponding leaf group. Then the fundamental group of the graph of groups   $G=\pi_{1}(\mathcal{G},S_{m})$ is $\mathbb{Z}$-linear.
\end{theorem}
\begin{proof}
	Let $F$ be the free group on the generators $\{t_{1}, t_{2},\dots,t_{N_{0}}\}$ that is associated to
	the central vertex and let $K_{1},\dots,K_{m}$ be the groups associated to the external vertices
	(leaves).
	For each edge $e_{\iota}$ choose a generator $c_{\iota}$ of its edge group and put $w_{\iota}=\alpha_{e_{\iota}}(c_{\iota})\in F$. Since $\omega_{e_{\iota}}(c_{\iota})$ is primitive in $K_{\iota}$, we can extend it to a basis of $K_{\iota}$: let $\{a_{0},a_{1},\dots,a_{N_{1}}\}$ be a basis of $K_{1}$ with $a_{0}=\omega_{e_{1}}(c_{1})$, let $\{b_{0},b_{1},\dots,b_{N_{2}}\}$ be a basis of $K_{2}$ with $b_{0}=\omega_{e_{2}}(c_{2})$, and so on, up to a basis $\{z_{0},z_{1},\dots,z_{N_{m}}\}$ of $K_{m}$ with $z_{0}=\omega_{e_{m}}(c_{m})$. Then
	$G$ is the fundamental group of the following star of groups, in which the edge $e_{\iota}$ identifies the first basis element of $K_{\iota}$ with $w_{\iota}$. 

	\begin{center}
		
		\begin{tikzpicture}[scale=1.02]

			\draw[thick,->] (0,-0.2) -- (2.9,0.2) node[midway,sloped,above] {\footnotesize{$w_{2}=b_{0}$}};
			\draw[thick,->] (0,-0.2) -- (1.7,2) node[midway,sloped,above] {\footnotesize{$w_{1}=a_{0}$}};
			\draw[thick,->] (0,-0.2) -- (1.4,-2.6) node[midway,sloped,below] {\footnotesize{$w_{m}=z_{0}$}};
			\fill (0cm,-0.2cm) circle (1.5pt); 
			\fill (2.9cm,0.2cm) circle (1.5pt); 
			\fill (1.7cm,2cm) circle (1.5pt); 
			\fill (1.4cm,-2.6cm) circle (1.5pt); 
			\node at (-0.4,-0.3) (nodeP) {\large{$F$}};
			\node at (3.4,0.1) (nodeP) {\large{$K_{2}$}};
			\node at (2.2,2.1) (nodeP) {\large{$K_{1}$}};
			\node at (1.9,-2.7) (nodeP) {\large{$K_{m}$}};
			\fill (1.2cm,-1.2cm) circle (1pt); 
			\fill (1.4cm,-0.85cm) circle (1pt); 
			\fill (1.55cm,-0.5cm) circle (1pt); 
		\end{tikzpicture}
		
	\end{center}
	\vspace{2mm}
	where $w_{\iota}$, $\iota\in\{1,\dots,m\}$, are non-trivial reduced words in $\{t_{1}, t_{2},\dots,t_{N_{0}}\}$.
	
	\par Replacing the monomorphism $\alpha_{e_{\iota}}$ by its composition with an inner automorphism of $F$ does not change the fundamental group of the graph of groups up to isomorphism (see \cite{Serre,Bass}). Hence we may assume that each $w_{\iota}$ is a cyclically reduced word in $\{t_{1},\dots,t_{N_{0}}\}$. Write $w_{1}=s_{1}s_{2}\cdots s_{L}$, where $s_{\lambda}\in\{t_{1}^{\pm1},\dots,t_{N_{0}}^{\pm1}\}$ are not necessarily distinct and the word is reduced and cyclically reduced. Since $a_{0}=w_{1}$, by using Tietze transformations we can remove $a_{0}$ from the generating set of $G$, and the relations $[a_{0},a_{i}]=1$ become $w_{1}a_{i}w_{1}^{-1}=a_{i}$, $i=1,\dots,N_{1}$. We introduce auxiliary generators by Tietze transformations
	$$q_{i,0}\defeq a_{i},\qquad q_{i,\mu}\defeq s_{L-\mu+1}\,q_{i,\mu-1}\,s_{L-\mu+1}^{-1},\qquad \mu=1,\dots,L-1,$$
	so that the relation $w_{1}a_{i}w_{1}^{-1}=a_{i}$ is replaced by the $L$ relations
	$$s_{L-\mu+1}\,q_{i,\mu-1}\,s_{L-\mu+1}^{-1}=q_{i,\mu},\qquad \mu=1,\dots,L,$$
	where the second index of $q_{i,\mu}$ is taken modulo $L$ (so $q_{i,L}=q_{i,0}=a_{i}$). Moreover, since $[a_{i},a_{i'}]=1$ for all $i,i'$, the relations $[q_{i,\mu},q_{i',\mu}]=1$ hold in $G$ for every $\mu$, and we may add them to the presentation. Let $\bar{K}_{1}$ be the right-angled Artin group with generators $q_{i,\mu}$, $1\leq i\leq N_{1}$, $0\leq\mu<L$, and relators $[q_{i,\mu},q_{i',\mu}]$.

	\par Let $t\in\{t_{1},\dots,t_{N_{0}}\}$. If $s_{\lambda}=t^{\epsilon}$ and $\mu=L-\lambda+1$, the corresponding relation says that conjugation by $t$ maps $q_{i,\mu-1}$ to $q_{i,\mu}$ if $\epsilon=1$, and $q_{i,\mu}$ to $q_{i,\mu-1}$ if $\epsilon=-1$, for every $i$. Hence $t$ acts by conjugation as a partial map $\pi_{t}$ of the index set $\{0,\ldots,L-1\}$ on the second index, uniformly in $i$. This partial map is injective: two of its arrows could have the same target or the same source only if two cyclically consecutive letters of $w_{1}$ were $t^{-1}t$ or $tt^{-1}$, which is excluded since $w_{1}$ is cyclically reduced. Therefore $\pi_{t}$ extends to a permutation $\bar{\pi}_{t}$ of the index set $\{0,\ldots,L-1\}$, and the map $q_{i,\mu}\mapsto q_{i,\bar{\pi}_{t}(\mu)}$ is a permutation of the generators of $\bar{K}_{1}$ which preserves its set of relators, i.e. an automorphism of the defining graph of $\bar{K}_{1}$, extending the action of $t$.

	If a leaf group $K_\iota$ has rank one, it contributes no auxiliary generators, and the corresponding group $\bar K_\iota$ is trivial.

	\par Working in the same way for the rest of the vertex groups $K_{2},\dots,K_{m}$, we obtain right-angled Artin groups $\bar{K}_{2},\dots,\bar{K}_{m}$, and $G$ is the multiple HNN-extension of the right-angled Artin group $\bar{K}=\bar{K}_{1}*\cdots *\bar{K}_{m}$ with stable letters $t_{1},\dots,t_{N_{0}}$, whose actions extend to automorphisms of the defining graph of $\bar{K}$ (acting on each $\bar{K}_{\iota}$ as above, and trivially on $\bar{K}_{\iota}$ if $t$ does not occur in $w_{\iota}$). Therefore, by Lemma \ref{Lemma2lin}, $G$ is $\mathbb{Z}$-linear.
\end{proof}
\section{Amalgamated products of free abelian groups}
In this section, using a slightly different technique, we prove that the amalgamated free product of two finitely generated free abelian groups is linear over $\Z$, and in fact virtually a right-angled Artin group (Remark \ref{virtRAAG}). The linearity is presumably known, although we have been unable to locate a proof in the literature. For the convenience of the reader, we provide a complete proof, for which we will need the following Lemma.
\begin{lemma}\label{Ker}[see \cite{Wehrfritz}]\ \\
	Let $G$ be a group and $\phi_{i}:G\to L_{i}$, $i=1,\dots,n$, homomorphisms to $\mathbb{Z}$-linear groups $L_{1},\dots,L_{n}$ such that $ \bigcap_{i=1}^{n}Ker\phi_{i}=\{1\} $. Then $G$ is $\mathbb{Z}$-linear.
\end{lemma}
\begin{theorem}\label{amalfreeprod}
	Let $G$ be the amalgamated free product of the finitely generated free abelian groups $G_{1}$ and $G_{2}$. Then $G$ is $\mathbb{Z}$-linear.
\end{theorem}
\begin{proof}
	Let $G=G_{1}*_{C}G_{2}$, where $C$ is identified with a subgroup of $G_{1}$ and with a subgroup of $G_{2}$. Since $G_{1}$ and $G_{2}$ are abelian, every element of $C$ commutes with $G_{1}$ and with $G_{2}$, and hence with $G=\langle G_{1},G_{2}\rangle$. Thus $C$ is a central (in particular normal) subgroup of $G$, and $$G/C\cong (G_{1}/C)*(G_{2}/C).$$ Let $\rho:G\to G/C$ be the quotient map.

	\textit{Step 1: $G/C$ is $\mathbb{Z}$-linear.} Each $G_{\iota}/C$, $\iota=1,2$, is a finitely generated abelian group, so $G_{\iota}/C=P_{\iota}\times T_{\iota}$ with $P_{\iota}$ free abelian and $T_{\iota}$ finite. Let $\pi:G/C\to T_{1}\times T_{2}$ be the homomorphism induced by the projections $G_{\iota}/C\to T_{\iota}$, and let $M=Ker\pi$, a subgroup of finite index in $G/C$. Every element of finite order of a free product is conjugate into one of the factors (see \cite{Serre}), and $\pi$ is injective on the torsion subgroup $T_{\iota}$ of each factor; hence $M$ is torsion-free. By the Kurosh subgroup theorem (see \cite{Serre} or \cite{Bog}), and since $M$ is torsion-free and finitely generated, $$M=F*M_{1}*\cdots *M_{s},$$ where $F$ is a finitely generated free group and each $M_{j}=M\cap g_{j}(G_{\iota_{j}}/C)g_{j}^{-1}$, for suitable $g_{j}\in G/C$ and $\iota_{j}\in\{1,2\}$, is a finitely generated free abelian group. Such a free product is a right-angled Artin group (its defining graph is a disjoint union of complete graphs), hence $\mathbb{Z}$-linear. As $M$ has finite index in $G/C$, the group $G/C$ is $\mathbb{Z}$-linear.

	\textit{Step 2: the abelianization is injective on $C$.} Let $\alpha:G\to G_{\mathrm{ab}}$ be the abelianization map, and write $i_{\iota}:C\hookrightarrow G_{\iota}$ for the two inclusions. We have
	\[
	G_{\mathrm{ab}}\cong (G_1\times G_2)/D,
	\qquad D=\{(i_1(c),i_2(c)^{-1}):c\in C\}.
	\]
	If $c\in C$ satisfies $\alpha(c)=1$, then $(i_1(c),1)\in D$, so $(i_1(c),1)=(i_1(d),i_2(d)^{-1})$ for some $d\in C$. Injectivity of $i_2$ gives $d=1$, and injectivity of $i_1$ then gives $c=1$. Thus $\alpha|_C$ is injective. Moreover, $G_{\mathrm{ab}}$ is finitely generated abelian and hence $\mathbb Z$-linear: its free abelian part is $\mathbb Z$-linear, and its finite part embeds in an integral permutation-matrix group by its regular action.

	\textit{Step 3.} We have
	\[
	\ker\alpha\cap\ker\rho=[G,G]\cap C=\{1\},
	\]
	so the homomorphism $G\to G_{\mathrm{ab}}\times(G/C)$, $g\mapsto(\alpha(g),\rho(g))$, is injective. Both factors are $\mathbb{Z}$-linear by Steps 1 and 2, so Lemma \ref{Ker} proves the assertion.
\end{proof}

\begin{remark}\label{virtRAAG}
	In fact, $G$ has a finite-index subgroup which is a finitely generated right-angled Artin group. We use the notation of Step 1. Let $\widetilde{M}=\rho^{-1}(M)$, a subgroup of finite index in $G$, so that $1\to C\to\widetilde{M}\to M\to1$ is a central extension. Choose $\tilde{g}_{j}\in G$ with $\rho(\tilde{g}_{j})=g_{j}$. Since $C\leq G_{\iota_{j}}$, we have $\rho^{-1}(g_{j}(G_{\iota_{j}}/C)g_{j}^{-1})=\tilde{g}_{j}G_{\iota_{j}}\tilde{g}_{j}^{-1}$, which is abelian. Hence $\rho^{-1}(M_{j})$ is an abelian group mapping onto the free abelian group $M_{j}$, so there is a homomorphism $\sigma_{j}:M_{j}\to\rho^{-1}(M_{j})$ with $\rho\sigma_{j}=\mathrm{id}$. Choosing preimages of a free basis of $F$ gives a homomorphism $\sigma_{F}:F\to\widetilde{M}$ with $\rho\sigma_{F}=\mathrm{id}$. By the universal property of free products we obtain a homomorphism $\sigma:M\to\widetilde{M}$ with $\rho\sigma=\mathrm{id}_{M}$. Then $\widetilde{M}=C\,\sigma(M)$, $C\cap\sigma(M)=\{1\}$ because $\rho$ is injective on $\sigma(M)$, and $C$ is central; therefore $\widetilde{M}\cong C\times M$. Since $C$ is a finitely generated free abelian group and $M$ is a right-angled Artin group, $C\times M$ is a finitely generated right-angled Artin group (its defining graph is the join of a complete graph on a basis of $C$ with the defining graph of $M$).
\end{remark}
\section{Generalized Baumslag--Solitar groups}\label{GBS}
A fundamental group of graphs of groups where the vertex and edge groups are infinite cyclic groups is called \emph{generalized Baumslag--Solitar group} or \emph{GBS-group} for short. GBS-groups have been the subject of recent research (see for example \cite{Fore,LevittA,Robi}). Higher-rank analogues, with vertex and edge groups isomorphic to $\mathbb{Z}^{n}$, have also been studied recently; in particular, Lopez de Gamiz Zearra and Shepherd \cite{LdGZShepherd} classified their residual-finiteness, subgroup-separability and cyclic-subgroup-separability properties.

Let $G$ be a group. A $G$-tree is a tree with a $G$-action by automorphisms, without inversions. A $G$-tree is minimal if there is no proper $G$-invariant subtree and cocompact if the quotient graph is finite. The quotient of the action is the graph of groups with vertex groups the vertex stabilizers and edge groups the edge stabilizers.

Given a $G$-tree $X$, an element $g\in G$ is called elliptic if it stabilizes a vertex of $X$ and is hyperbolic otherwise. If $g$ is hyperbolic then there is a unique $\langle g\rangle$-invariant line in $X$, the axis of $g$, on which $g$ acts as a translation. A subgroup $H$ of $G$ is elliptic if it fixes a vertex.

A generalized Baumslag--Solitar tree is a $G$-tree whose vertex stabilizers and edge stabilizers are all infinite cyclic. The group $G$ that arises is called generalized Baumslag--Solitar group (GBS group). The quotient graph of groups has all vertex and edge groups isomorphic to $\Z$ and the inclusion maps are multiplication by various non-zero integers.

The group $G$ is torsion free since any torsion group intersects each conjugate of a vertex group trivially which implies that it is free and therefore trivial.  Moreover, all non-trivial elliptic elements are commensurable.

GBS-groups are presented by labeled graphs, that is finite graphs where each oriented edge $e$ has label $\lambda_e\neq 0$. If $\Gamma$ is a finite graph with vertex and edge groups isomorphic to $\Z$ then we have a graph of groups with fundamental group a GBS group.  A presentation for a GBS-group $G$ can be obtained as follows. Choose a maximal tree $T\subset \Gamma$. There is one generator $g_v$ for every vertex of $\Gamma$ and one generator $t_e$ for every edge in $\Gamma\setminus T$ (non-oriented). Each non-oriented edge contributes  a relation in the presentation. The relations are
$$(g_{\alpha(e)})^{\lambda_e}=(g_{\alpha(\overline{e})})^{\lambda_{\overline{e}}}\ \ \mbox{if}\ \ \  e\in T$$
$$t_e(g_{\alpha(e)})^{\lambda_e}t_e^{-1}=(g_{\alpha(\overline{e})})^{\lambda_{\overline{e}}}\ \   \mbox{if}\ \ e\in\Gamma\setminus T$$
The above is the standard presentation of $G$ and the generating set is the standard generating set associated to $\Gamma$ and $T$.

Notice that the group $G$ represented by $\Gamma$ does not change if we change the sign of all labels near a vertex $v$ or the labels carried by a non-oriented edge. These are called admissible sign changes. Notice also that there may be infinitely many labelled graphs representing the same $G$.

A GBS-group is elementary if it is isomorphic to $\Z$, $\Z^2$ or the Klein bottle $K=\<a,t\mid tat^{-1}=a^{-1}\>=\<x,y\mid x^2=y^2\>.$ These groups have special properties.

If a GBS-group $G$ is non-elementary then there is a modular homomorphism (introduced by Kropholler in \cite{Krop}) $\Delta_G: G\ar {\mathbb Q}^*$ associated to $G$.  Given any non-trivial elliptic element $a$, there is a non-trivial relation $ga^mg^{-1}=a^n$. Then $\Delta_G(g)=m/n$.  Notice that $m,n$ may depend on the choice of $a$ but $m/n$ does not.  If $g$ is elliptic, $\Delta_G(g)=1$.  Notice that the existence of a relation $th^mt^{-1}=h^n$ with $|m|\neq|n|$ implies that $h$ is elliptic since it satisfies $|m|\ell(h)=|n|\ell(h)$ where $\ell(h)$ is the translation length of $h$.  Remember that the translation length of an element $g$ is $\ell(g)=\liminf_n\frac{d(x,g^nx)}{n}$ which obviously implies that conjugate elements have the same translation length.

Since $\Delta_G$ is trivial on each vertex group, it factors through the natural quotient $q:G\to\pi_1(\Gamma)$ obtained by killing the vertex groups. Write $\Delta_G=\overline{\Delta}_G\circ q$. With the path convention of Section~2, traversal of an oriented edge $e$ is represented by $t_e^{-1}$. Thus, if $\gamma\in\pi_1(\Gamma)$ is represented by a loop $e_1\cdots e_m$, then
\[
\overline{\Delta}_G(\gamma)=\prod_{i=1}^{m}\frac{\lambda_{\overline{e_i}}}{\lambda_{e_i}}.
\]
In particular, $\Delta_G$ is trivial if $\Gamma$ is a tree.

Also $\Delta_G$ is trivial if and only if the centre of $G$ is non-trivial. In this case the centre is cyclic and it contains only elliptic elements. A non-elementary GBS is \emph{unimodular} if the image of $\Delta_G$ is contained in $\{1,-1\}$. This is equivalent to $G$ having a normal infinite cyclic subgroup and also to $G$ being virtually $F_n\times{\mathbb \Z}$ for some $n\ge 2$ (see \cite{LevittB}).

Let $e\in E(\Gamma)\setminus E(T)$, and let $f_1\cdots f_r$ be the unique oriented path in $T$ from $v=\alpha(e)$ to $u=\omega(e)$. Put
\[
p_1(e)=\prod_{j=1}^{r}\lambda_{f_j},\qquad
p_2(e)=\prod_{j=1}^{r}\lambda_{\overline{f_j}},
\]
where an empty product is $1$. The tree-edge relations give $g_v^{p_1(e)}=g_u^{p_2(e)}$. We say that $e$ is \emph{$T$-dependent} or \emph{skew $T$-dependent} if $$\frac{p_{1}(e)}{p_{2}(e)}=\frac{\lambda_{e}}{\lambda_{\bar{e}}} \mbox{ or} -\frac{\lambda_{e}}{\lambda_{\bar{e}}}.$$
respectively. If $e$ is a loop, then by convention, $e$ is $T$-dependent or skew $T$-dependent if $\lambda_{e} =\lambda_{\bar{e}}$ or $-\lambda_{\bar{e}}$,
respectively.

If every non-tree edge is $T$-dependent, $G$ is called tree dependent. If every non-tree edge is
$T$-dependent or skew $T$-dependent, with at least one non-tree edge of the latter type, then $G$ is said to be skew tree dependent. It turns out that the properties of tree dependence and skew tree dependence
are independent of the choice of maximal subtree $T$.

Let $Z(G)$ and $C(G)$ denote the centre and the maximum cyclic normal subgroup of $G$ (called the \emph{cyclic radical}) respectively. Clearly $C(G)$ contains $Z(G)$. 
\begin{proposition}[see \cite{Delga}]\label{DRTprop}
	Let $G$ be a non-elementary GBS-group. Then the following statements are equivalent.
	\begin{enumerate}
		\item $G$ is tree dependent, respectively, skew tree dependent.
		\item $Im(\Delta_G)=\{1\}$, respectively, $Im(\Delta_G)=\{-1,1\}.$
		\item $Z(G)\neq 1$, respectively, $1=Z(G)<C(G).$
	\end{enumerate}
\end{proposition}
As stated above, in the case where the image of $\Delta_G$ is contained in $\{1,-1\}$, the group $G$ is virtually $F_n\times{\mathbb \Z}$ for some $n\ge 2$ (see \cite{LevittB}). Since $F_{2}$, and hence $F_{n}$, embeds in $SL_{2}(\Z)$, and $\Z$ embeds in $SL_{2}(\Z)$ (generated by any matrix of infinite order), $F_{n}\times\Z$ embeds in $GL_{4}(\Z)$ via block-diagonal matrices; therefore $G$ is $\Z$-linear, being a finite extension of a $\Z$-linear group. For the converse we use the following result of Button \cite[p.~375]{ButtonBalanced}, for which we include a self-contained matrix proof.

\begin{lemma}\label{noBS}
	Let $g,t\in GL_{k}(\Z)$, where $g$ has infinite order, and let $M,N$ be non-zero integers such that $tg^{M}t^{-1}=g^{N}$. Then $|M|=|N|$. In particular, no Baumslag--Solitar group $BS(M,N)$ with $|M|\neq|N|$ is $\Z$-linear.
\end{lemma}
\begin{proof}
	Assume that $|M|\neq|N|$. Let $\lambda_{1},\dots,\lambda_{k}$ be the eigenvalues of $g$ in $\C$ (with multiplicities), and let $\rho=\max_{j}|\lambda_{j}|$ and $\sigma=\min_{j}|\lambda_{j}|>0$. Since $g^{M}$ and $g^{N}$ are conjugate, the multisets $\{\lambda_{j}^{M}\}$ and $\{\lambda_{j}^{N}\}$ coincide. Replacing $(M,N)$ by $(-M,-N)$ if necessary, we may assume that $M>0$. If $N>0$, comparing maximal and minimal moduli gives $\rho^{M}=\rho^{N}$ and $\sigma^{M}=\sigma^{N}$, so $\rho=\sigma=1$ since $M\neq N$. If $N<0$, we get $\rho^{M}=\sigma^{N}$ and $\sigma^{M}=\rho^{N}$, i.e. $M\log\rho=N\log\sigma$ and $M\log\sigma=N\log\rho$; adding and subtracting these equalities and using $M-N\neq0$ and $M+N\neq 0$, we obtain $\log\rho=-\log\sigma$ and $\log\rho=\log\sigma$, so again $\rho=\sigma=1$. Hence all eigenvalues of $g$ have modulus $1$. Since the characteristic polynomial of $g$ is monic with integer coefficients, its set of roots is closed under Galois conjugation, so every $\lambda_{j}$ is an algebraic integer all of whose conjugates have modulus $1$; by Kronecker's theorem (see \cite{Kronecker}) every $\lambda_{j}$ is a root of unity. Hence there is $e\geq1$ such that $u=g^{e}$ is unipotent, and $u\neq I$ since $g$ has infinite order. Moreover $tu^{M}t^{-1}=(tg^{M}t^{-1})^{e}=u^{N}$.

	Since $g$ has infinite order, $k\geq2$. Let $X=\log u=\sum_{j=1}^{k-1}\frac{(-1)^{j+1}}{j}(u-I)^{j}$. Then $X$ is a non-zero nilpotent matrix with rational entries, $dX$ has integer entries, where $d=\operatorname{lcm}(1,\dots,k-1)$, and $\log(u^{s})=sX$ for every $s\in\Z$. Since conjugation commutes with $\log$, we get $M\,tXt^{-1}=\log(tu^{M}t^{-1})=\log(u^{N})=NX$, so $tXt^{-1}=cX$ with $c=N/M$, and therefore $t^{\ell}Xt^{-\ell}=c^{\ell}X$ for every $\ell\in\Z$. As $t^{\pm1}\in GL_{k}(\Z)$, the matrix $c^{\ell}(dX)=t^{\ell}(dX)t^{-\ell}$ has integer entries for every $\ell\in\Z$. Write $c=p/q$ with $\gcd(p,q)=1$ and $q>0$, and let $x\neq0$ be an entry of $dX$. Since $|c|\neq1$, either $q\geq2$ or $|p|\geq2$. In the first case, $c^{\ell}x\in\Z$ for all $\ell>0$ forces $q^{\ell}\mid x$ for all $\ell>0$, and in the second case $c^{-\ell}x\in\Z$ for all $\ell>0$ forces $p^{\ell}\mid x$ for all $\ell>0$; both are impossible since $x\neq0$. This contradiction shows that $|M|=|N|$.

	Finally, $BS(M,N)=\langle a,t\mid ta^{M}t^{-1}=a^{N}\rangle$ is an HNN-extension of $\langle a\rangle\cong\Z$, so $a$ has infinite order, and the first statement applies to the images of $a$ and $t$ under any faithful representation of $BS(M,N)$ into $GL_{k}(\Z)$.
\end{proof}

\begin{proposition}
	A non-elementary GBS-group $G$ is $\mathbb{Z}$-linear if and only if $Im(\Delta_G)\subseteq \{-1,1\}$, i.e. (by Proposition \ref{DRTprop}) if and only if $G$ is tree dependent or skew tree dependent.
\end{proposition}
\begin{proof}
	One direction was shown above. Conversely, assume that $Im(\Delta_{G})\not\subseteq\{-1,1\}$ and let $g\in G$ with $\Delta_{G}(g)\neq\pm1$. Let $a\in G$ be a non-trivial elliptic element. As recalled above, there is a relation $ga^{M}g^{-1}=a^{N}$ with $M,N\neq0$ and $\Delta_{G}(g)=M/N$, so $|M|\neq|N|$. Since $G$ is torsion-free, $a$ has infinite order, so by Lemma \ref{noBS} the group $G$ admits no faithful representation into $GL_{k}(\Z)$ for any $k$, i.e. $G$ is not $\Z$-linear.
\end{proof}
\hspace{-5mm}The linearity of elementary GBS-groups is trivial: $\Z$ and $\Z^{2}$ are $\Z$-linear, and the Klein bottle group contains $\Z^{2}$ as a subgroup of index $2$.

\begin{remark}
	The above Proposition can also be derived from a result of Levitt \cite{LevittA}, who showed that a GBS-group is residually finite if and only if it is solvable or virtually $F\times\Z$ with $F$ free. Indeed, a finitely generated linear group is residually finite (see \cite{Wehrfritz}), and a solvable subgroup of $GL_{k}(\Z)$ is polycyclic (see \cite{Segal}). A non-elementary solvable GBS-group is isomorphic to a Baumslag--Solitar group $BS(1,n)$ with $|n|\geq2$, since a non-elementary GBS-group either is such a group or contains a non-abelian free group (see \cite[Section 2]{LevittB}). The group $BS(1,n)=\langle a,t\mid tat^{-1}=a^{n}\rangle$ is not polycyclic: since $(t^{-1}at)^{n}=a$, the normal closure of $a$ is the strictly increasing union $\bigcup_{j\geq0}\langle t^{-j}at^{j}\rangle$ of infinite cyclic groups, which is isomorphic to the additive group of $\Z[1/n]$ and is not finitely generated. Hence a $\Z$-linear non-elementary GBS-group is virtually $F\times\Z$, so it is unimodular. Lemma \ref{noBS} gives a self-contained proof of this direction.
\end{remark}
\section*{Acknowledgments}
This paper has been developed as part of my PhD research. I am deeply grateful to my supervisor, V. Metaftsis, whose insights were instrumental in the development of this manuscript.

\end{document}